\documentclass[cmbright]{my_mmaauth}

\usepackage{moreverb}

\usepackage[colorlinks,bookmarksopen,bookmarksnumbered,citecolor=red,urlcolor=red]{hyperref}

\usepackage{amsmath,amssymb}
\usepackage[all]{xy}
\usepackage{subfig}
\usepackage{pdfsync}
\usepackage{graphicx}
\usepackage{epstopdf}
\usepackage{morefloats}

\newtheorem{theorem}{Theorem}

\newtheorem{remark}[theorem]{Remark}

\def\R{\mathbb{R}}

\newenvironment{proof}{\paragraph{Proof:}}{\hfill$\square$}

\begin{document}

\runninghead{D.~Rocha, C.~J.~Silva, D.~F.~M.~Torres}
	
\title{Stability and Optimal Control of a Delayed HIV Model}

\author{Diana Rocha\affil{a}, 
Cristiana J. Silva\affil{a} 
and Delfim F. M. Torres\affil{a}\corrauth}

\address{\affilnum{a}Center for Research and Development in Mathematics and Applications (CIDMA),
Department of Mathematics, University of Aveiro, 3810-193 Aveiro, Portugal}

\corraddr{Department of Mathematics,
University of Aveiro,
3810-193 Aveiro, Portugal.
Email: delfim@ua.pt}

% ---------------------------------------------

\begin{abstract}
We propose and investigate a delayed model that
studies the relationship between HIV and the immune system 
during the natural course of infection and in the context 
of antiviral treatment regimes. Sufficient criteria 
for local asymptotic stability of the infected 
and viral free equilibria are given.
An optimal control problem with time delays both 
in state variables (incubation delay) and control 
(pharmacological delay) is then formulated and analyzed, 
where the objective consists to find the optimal treatment 
strategy that maximizes the number of uninfected $CD4^{+}$ 
T cells as well as CTL immune response cells, keeping 
the drug therapy as low as possible.
\end{abstract}

\MOS{34C60; 49K15; 92D30}


\maketitle

\vspace{-6pt}

% ----------------------
	
\section{Introduction}

The study of mathematical models for human immunodeficiency 
virus (HIV) infection is a subject of strong current interest, 
both at population and cell levels (see, e.g., 
\cite{Li,MR3392642,WodarzNowak} and references cited therein).
Based on the model of \cite{CulshawRuanSpiteri2004},
in this work we analyze a mathematical model that studies the relationship 
between HIV and the immune system during the natural course of infection
and in the context of antiviral treatment regimes.
The model considers three variables: uninfected $CD4^{+}$ T cells,
denoted by $x$; infected $CD4^{+}$ T cells, denoted by $y$; 
and CTL effectors (immune response cells), denoted by $z$.
According to \cite{Arnaout:2000}, the viral load is assumed 
to be proportional to the level of infected cells.  
Uninfected $CD4^{+}$ T cells are produced at a rate $\lambda$, 
die at a rate $d$, and become infected at a rate $\beta$. 
Infected cells decay at a rate $a$ and are killed by CTL effectors 
at a rate $p$. Proliferation of the CTL population is given 
by $cxyz$ and is proportional to both virus load ($y$) 
and the number of uninfected $CD4^{+}$ T cells ($x$). 
CTL effectors die at a rate $h$. Mathematically, the
model \cite{CulshawRuanSpiteri2004} is described by
\begin{equation}
\label{eq:model}
\begin{cases}
\dot{x}={\lambda}-dx-{\beta}xy, \\
\dot{y}={\beta}xy-ay-pyz, \\
\dot{z}=cxyz-hz.
\end{cases}
\end{equation}

Time delays play an important role in the dynamics of HIV infection: 
see, e.g., \cite{Mittler:HIVdelay:Math:Bio:1998,NelsonMurray:delay:HIV:2000} 
and references therein. For this reason, in this work we introduce a discrete
time-delay into the model \eqref{eq:model}, which represents 
the incubation period, that is, the time between the new infection 
of a $CD4^{+}$ T cell and the time it becomes infectious (Section~\ref{sec:2}).
We prove local asymptotic stability of the viral free and infected equilibriums, 
for any time delay (Section~\ref{sec:3}). 

Optimal control theory has been applied with success to 
epidemiology HIV models: see, e.g., 
\cite{CulshawRuanSpiteri2004,Haffat:Yousfi:OCHIVdelay:2012} 
and references therein. However, epidemiology optimal control problems
with delays in both state and control variables are a rarity.
For one such optimal control problem, of a tuberculosis model, 
we refer the reader to the recent work of
Silva, Maurer and Torres \cite{MyID:353}.
Here, we propose and solve 
a HIV optimal control problem, 
with delays in both state and control variables,
where the objective is to find 
the optimal treatment strategy that maximizes the number 
of $CD4^{+}$ T cells, as well as CTL immune response cells, 
keeping the drug/chemotherapy strength, which depends
on a pharmacological delay, as low as possible
(Section~\ref{sec:OC}). The stability and optimal control
results of Sections~\ref{sec:3} and \ref{sec:OC}
are then illustrated through numerical simulations 
in Section~\ref{sec:numerical}. We end with 
Section~\ref{sec:conc:fw} of conclusions and future work.

% ----------------------

\section{Model with time delay $\tau$}
\label{sec:2}

In epidemiological literature, 
a latent or incubation period is often modeled by
incorporating it as a delay effect \cite{Kaddar}. 
We consider the following delayed model,  
where $\tau> 0$ represents the incubation period:
\begin{equation}
\label{model:delay}
\begin{cases}
\dot{x}(t)={\lambda}-dx(t)-{\beta}x(t)y(t),\\
\dot{y}(t)={\beta}x(t-\tau)y(t-\tau)-ay(t)-py(t)z(t),\\
\dot{z}(t)=cx(t)y(t)z(t)-hz(t).
\end{cases}
\end{equation}
The initial conditions for system \eqref{model:delay} are
\begin{equation}
\label{eq:init:cond:moddelay}
x(\theta) = \varphi_1(\theta), 
\quad y(\theta) = \varphi_2(\theta), 
\quad z(\theta) = \varphi_3(\theta), 
\end{equation}
$-\tau \leq \theta \leq 0$, where 
$\varphi=\left(\varphi_1, \varphi_2, \varphi_3 \right)^T \in C$ 
with $C$ the Banach space $C \left([-\tau, 0], {\mathbb{R}}^3 \right)$ 
of continuous functions mapping the interval $[-\tau, 0]$ into ${\mathbb{R}}^3$.
The usual local existence, uniqueness and continuation results apply
\cite{Hale_Lunel_book1993,YKuang_1993}. Therefore, there exists 
a unique solution $\left( x(t), y(t), z(t)\right)$ of \eqref{model:delay}
with initial conditions \eqref{eq:init:cond:moddelay}, 
for all time $t \geq 0$. From biological meaning, we further assume 
the initial functions to be non-negative, that is,
\begin{equation}
\label{eq:non:neg}
\varphi(\theta) \geq 0 \quad \text{for}  
\quad \theta \in [-\tau, 0], \quad i = 1, \ldots, 3. 
\end{equation}
From \cite[Theorem 2.1]{ZhuZou:DelayHIV:DCDS:2009}, it follows that all solutions 
of \eqref{model:delay} satisfying \eqref{eq:init:cond:moddelay} with \eqref{eq:non:neg}
are bounded for all time $t \geq 0$, which ensures not only local existence 
but also the existence of a solution for all time $t \geq 0$.

It is easy to see that system \eqref{model:delay} 
has an infection-free equilibrium 
\begin{equation}
\label{eq:E0}
E_0=\left(\frac{\lambda}{d},0,0\right), 
\end{equation}
which is the only biologically meaningful 
equilibrium if $\beta < \frac{d a}{\lambda}$. 
Let $\beta > \frac{d a}{\lambda}$ and assume that $\lambda c - \beta h > 0$. 
If $\beta < \frac{a c d}{\lambda c - \beta h}$, then system \eqref{model:delay} 
has a unique infected equilibrium $E_1$ given by
\begin{equation}
\label{eq:E1}
E_1=\Biggl( \frac{a}{\beta}, \frac{\lambda \beta - d a}{\beta a}, 0 \Biggr).
\end{equation}
Assume that $\lambda c - \beta h > 0$. Whenever $\beta > \frac{a c d}{\lambda c - \beta h}$, 
the unique infected equilibrium of system \eqref{model:delay} is given by
the CTL equilibrium 
\begin{equation}
\label{eq:E2}
E_2=\Biggl(  \frac{\lambda c - \beta h}{c d}, 
\frac{d h}{ \lambda c - \beta h}, 
\frac{\beta (\lambda c- \beta h)}{c d p} - \frac{ a}{p} \Biggr).
\end{equation}

% ---------------------------------------------------

\section{Local asymptotic stability}
\label{sec:3}

Consider the following coordinate transformation:
\begin{equation*}
X(t) = x(t)-\bar{x} \, ,\quad
Y(t) = y(t)-\bar{y} \, , \quad
Z(t) = z(t)-\bar{z} \, , 
\end{equation*}
where $(\bar{x},\bar{y},\bar{z})$ denotes any equilibrium 
of \eqref{model:delay}: $E_0$, $E_1$ or $E_2$.
The linearized system of \eqref{model:delay} is of form
\begin{equation}
\label{mod:linear}
\begin{cases}
\dot{X}(t)= (-d-\beta \bar{y})X(t) - \beta\bar{x}Y(t) \, ,\\
\dot{Y}(t)= (-a-p\bar{z})Y(t)+\beta\bar{y}X(t-\tau) +\beta\bar{x}Y(t-\tau) - p\bar{y}Z(t) \, , \\
\dot{Z}(t)= c\bar{y}\bar{z}X(t)+c\bar{x}\bar{z}Y(t)+(c \bar{x}\bar{y} -h)Z(t) \, .
\end{cases}
\end{equation}
We can express system \eqref{mod:linear} in matrix notation as follows: 
$$
\frac{d}{dt} 
\left( 
\begin{array}{c}
X(t) \\ 
Y(t) \\
Z(t) 
\end{array} 
\right) 
= A_{1} \left( \begin{array}{c}
X(t) \\ 
Y(t) \\
Z(t) 
\end{array} \right) 
+ A_{2} \left( 
\begin{array}{c}
X(t-\tau) \\ 
Y(t-\tau) \\
Z(t-\tau) 
\end{array}
\right),
$$
where $A_{1}$ and $A_{2}$ are the $3 \times 3$ matrices given by
$$
A_{1}=\left( \begin{array}{ccc}
-d-\beta\bar{y} & -\beta\bar{x} & 0  \\ 
0 & -a-p\bar{z} & -p\bar{y}   \\
c\bar{y}\bar{z} & c\bar{x}\bar{z} & c \bar{x}\bar{y}-h
\end{array} \right)
$$
and
$$
A_{2}=\left( \begin{array}{ccc}
0 & 0 & 0  \\ 
\beta\bar{y} & \beta\bar{x} & 0  \\
0 & 0 & 0 
\end{array} \right) \, . 
$$

%---------------------------------------------------

\subsection{Local stability of the infection-free equilibrium $E_0$ for any time delay $\tau$}

The characteristic equation of system \eqref{model:delay} is given by
\begin{equation}
\label{eq:charact:geral}
\Delta (\lambda_{1})=\det(\lambda_{1}Id-A_{1}-e^{-\lambda_{1}\tau}A_{2}) = 0,
\end{equation}
where $Id$ denotes the identity matrix of dimension 3. The following result holds.

\begin{theorem}
(i) If $\beta \lambda - ad < 0$, then the infection-free equilibrium $E_{0}$ \eqref{eq:E0}
is locally asymptotically stable for any time delay $\tau \geq 0$. 
(ii) If  $\beta \lambda - ad > 0$, then $E_0$ \eqref{eq:E0} 
is unstable for any time-delay $\tau \geq 0$.
(iii) If $\beta \lambda - ad = 0$, then a critical case occurs.
\end{theorem}

\begin{proof}
(i) The characteristic equation \eqref{eq:charact:geral} associated 
to the infection-free equilibrium is given by
\begin{equation}
\label{eq:charac:E0}
\frac{(d + \lambda_1)(h + \lambda_1)(ad + d\lambda_1 
	- \beta\lambda e^{-\lambda_1\tau})}{d}=0  \, .
\end{equation}
Assume that $\tau = 0$. In this case, the equation \eqref{eq:charac:E0} becomes
\begin{equation}
\label{eq:charac:E0:tau0}
\frac{(d + \lambda_1)(h + \lambda_1)(ad 
+ d\lambda_1 - \beta\lambda e^{-\lambda_1\tau})}{d}=0 \, .
\end{equation}
The roots of \eqref{eq:charac:E0:tau0} are $-d$, $-h$ and $\frac{\beta \lambda  - d a}{d}$, 
which have negative real part when $\beta \lambda - da  < 0$. 
Suppose that $\tau > 0$. To prove the stability of $E_0$ we use Rouch\'{e}'s theorem. 
Thus, we need to prove that the roots of the characteristic equation \eqref{eq:charac:E0} 
cannot intersect the imaginary axis, i.e., the characteristic equation cannot 
have pure imaginary roots. Suppose the contrary, i.e., that there exists 
a nonzero $w$ such that $b = w i$ is a solution of \eqref{eq:charac:E0}. 
Equation \eqref{eq:charac:E0} has two real negative solutions $-d$ and $-h$. 
Therefore, in what follows we just consider the term 
$\frac{ad + d\lambda_1 - \beta\lambda e^{-\lambda_1\tau}}{d}$. 
The complex $b = w i$ is a root of \eqref{eq:charac:E0} if  
$\frac{ad + d w i - \beta\lambda e^{-w i \tau}}{d} = 0$.
By using Euler's formula $\exp^{-i w \tau} = \cos(w \tau) -i \sin(w \tau)$,
and by separating real and imaginary parts, we have
\begin{equation*}
\begin{cases}
d a =\beta \lambda \cos(w \tau) \\
d w = - \beta \lambda \sin(w \tau) \, .
\end{cases}
\end{equation*}
Adding up the squares of both equations, we obtain that
$$ 
d^2 w^2 + d^2 a^2  - \lambda^2 \beta^2 = 0,
$$
that is, 
$$
w^2 =\frac{\lambda^2 \beta^2 - d^2 a^2}{d^2} \, .
$$
If $\beta \lambda - da <0$, then $w^2 < 0$, which is a contradiction. 
We just proved that the characteristic equation \eqref{eq:charac:E0} 
cannot have pure imaginary roots and the infection-free equilibrium 
$E_0$ is locally asymptotically stable for any strictly positive time delay.

(ii) Suppose now that $\beta \lambda - da > 0$. We know that 
the characteristic equation \eqref{eq:charac:E0} 
has two real negative roots: $\lambda_1 = -d$ and $\lambda_1 = -h$. 
Thus, we must check if the remaining roots of function 
$f(\lambda_1 ) := a + \lambda_1 - \frac{\beta\lambda e^{-\lambda_1\tau}}{d}$
have negative real parts. It is easy to see that $f(0) = a - \frac{\beta\lambda}{d} < 0$, 
since we are assuming $\beta \lambda - da > 0$. On the other hand, 
$\displaystyle \lim_{\lambda_1 \to + \infty} f(\lambda_1) = +\infty$. 
Therefore, by continuity of $f(\lambda_1)$, there is at least one 
positive root of the characteristic equation \eqref{eq:charac:E0}. 
Hence, we conclude that $E_0$ is unstable. 

(iii) Finally, we analyze the case $\beta \lambda - da = 0$. 
In this situation, the characteristic equation \eqref{eq:charac:E0} becomes 
\begin{equation}
\label{eq:charac:R0=1}
(d+\lambda_1)(h+\lambda_1)(\lambda_1 + a -\exp^{-\lambda_1 \,\tau} a) = 0 \, .
\end{equation}
To prove the stability we need to check again if all the roots 
of the above equation have negative real parts. Notice that $\lambda_1 = 0$, 
$\lambda_1 = -d$ and $\lambda_1 = -h$ are solutions of this equation, 
so we just need to prove that the remaining roots cannot have nonnegative 
real parts. Assuming that $\lambda_1 = u + w i $ with $u \geq 0 $ 
is a solution of the above equation, then 
$$
u+wi+a-\exp^{- \left( u+wi \right) \tau} a = 0 \, .
$$
By using Euler formula and separating the real and imaginary parts, we get 
\begin{equation*}
\begin{cases}
u + a = {{\rm e}^{-u\tau}}a\cos \left( w\tau \right) \\
w  = -{{\rm e}^{-u\tau}}a \sin \left( w\tau \right) \, .
\end{cases}
\end{equation*}
By adding up the squares of both equations, and using 
the fundamental trigonometric formula, we obtain
$$
(u + a )^2 + w^2 = \exp(-2u \tau) a^2 \leq a^2,
$$
which is a contradiction. This proves that $0$ is the unique 
root of \eqref{eq:charac:R0=1} that does not have negative real part. 
\end{proof}
	
%---------------------------------------------------

\subsection{Local stability of the infected equilibrium $E_1$ for any time delay $\tau$}

We now study the local stability of the 
the infected equilibrium $E_1$ \eqref{eq:E1}
for any incubation period $\tau$.

\begin{theorem}
Let $\beta \lambda - da > 0$ and assume that $\lambda c - \beta h > 0$. 
(i) If $\beta (\lambda c - \beta h) - acd < 0$, then the infected equilibrium 
$E_1$ is locally asymptotically stable for any time delay $\tau \geq 0$. 
(ii) If $\beta (\lambda c - \beta h) - acd > 0$, then 
$E_1$ is unstable for any time delay $\tau \geq 0$.
\end{theorem}

\begin{proof}
Let $\beta \lambda - da > 0$ and $\lambda c - \beta h > 0$. 
The characteristic equation \eqref{eq:charact:geral} 
at $E_1 =\Biggl( \frac{a}{\beta}, 
\frac{\lambda \beta - d a}{\beta a}, 0 \Biggr) $ is given by
\begin{equation}
\label{eq:charac:E1}
{\frac { \left( \lambda_1 \, \beta^{2}-\beta\,
\lambda\,c+{\beta}^{2}h+acd \right)  
\left( - \lambda_1^{2}a- \lambda_1 \,{a}^{2}
+\lambda_1 \,{a}^{2}{{\rm e}^{-\lambda_1 \,\tau}}
-\lambda\,\beta\,\lambda_1 -\lambda\,\beta\,a
+{{\rm e}^{-\lambda_1 \,\tau}}{a}^{2}d 
\right) }{a{\beta}^{2}}} = 0 \, .
\end{equation}
Note that $\lambda_1 = \frac {\beta\,\lambda\,c-{\beta}^{2}h-acd}{{\beta}^{2}}$ 
is a solution of \eqref{eq:charac:E1}. 
(i) If $\beta (\lambda c - \beta h) - acd < 0$, 
then $\lambda_1 = \frac {\beta\,\lambda\,c-{\beta}^{2}h-acd}{{\beta}^{2}}$ 
is a real negative root of the characteristic equation 
\eqref{eq:charac:E1} and we just need to analyze the equation
\begin{equation}
\label{eq:charac:E1:simpl}
\frac { \lambda_1^{2}a + \lambda_1 ({a}^{2} 
+ \lambda\,\beta) + \lambda\,\beta\,a 
- (\lambda_1 + d){a}^{2}{{\rm e}^{-\lambda_1 
\,\tau}} }{a{\beta}^{2}} = 0 \, .
\end{equation}
Consider $\tau = 0$. From equation \eqref{eq:charac:E1:simpl}, we have 
\begin{equation}
\label{eq:charac:E1:tau0}
{\frac {{\lambda_1}^{2}a + \lambda\,\beta\,\lambda_1 +\lambda\,
\beta\,a - {a}^{2}d}{a{\beta}^{2}}} = 0.
\end{equation}
Since $\beta \lambda - da > 0$,
it follows that $\frac{1}{\beta^2} > 0$, 
$\frac{\lambda}{a \beta} > 0$ 
and $\frac{\lambda\,\beta\,a - {a}^{2}d}{a{\beta}^{2}} 
= \frac{\lambda\, \beta  - a d}{{\beta}^{2}} > 0$. Therefore, 
from the Routh--Hurwitz criterion, it follows that all roots 
of \eqref{eq:charac:E1:tau0} have negative real part. 
Hence, $E_1$ is locally asymptotically stable for $\tau=0$.
Let $\tau > 0$. Suppose that \eqref{eq:charac:E1:simpl} 
has pure imaginary roots $\pm w i$. By replacing $\lambda_1$ 
in \eqref{eq:charac:E1:simpl} by $w i$, 
and separate the real and imaginary parts, 
we obtain
\begin{equation*}
\begin{cases}
-{w}^{2}a + \lambda\,\beta\,a 
= {a}^{2}d\cos \left( w\tau \right) 
+ w{a}^{2}\sin \left( w\tau \right)\\
w{a}^{2} + \lambda\,\beta\,w =  w{a}^{2}\cos\left( w\tau \right)
- {a}^{2}d\sin \left( w\tau \right).
\end{cases}
\end{equation*}
By adding up the squares of both equations, 
and using the fundamental trigonometric formula, 
we obtain that
\begin{equation*}
{a}^{2}{w}^{4}+{w}^{2}{\lambda}^{2}{\beta}^{2} 
= {a}^{2}(a^2 {d}^{2} -{\lambda}^{2}{\beta}^{2} ) \, ,
\end{equation*}
which is a contradiction since $\beta \lambda - da > 0$. 
Therefore, $a^2 {d}^{2} -{\lambda}^{2}{\beta}^{2} < 0$ 
and equation \eqref{eq:charac:E1:simpl} does not have 
pure imaginary roots. This implies that $E_1$ is locally 
asymptotically stable for any time delay $\tau > 0$. 
(ii) If $\beta > \frac{a c d}{\lambda c - \beta h}$,  
then the characteristic equation \eqref{eq:charac:E1} 
has a positive root and consequently the equilibrium 
$E_1$ is unstable for any time delay  $\tau \geq 0$.
\end{proof}

%---------------------------------------------------

\subsection{Local stability of the CTL equilibrium $E_2$}

The analysis of the local stability of the CTL equilibrium $E_2$
is more complex. Under some assumptions, the situation is clear
for $\tau = 0$: the infected equilibrium $E_2$ 
is locally asymptotically stable (see Theorem~\ref{thm:stab:E2}). 
However, for $\tau > 0$, the characteristic polynomial has pure imaginary roots 
and we are not able to conclude anything about the stability for an arbitrary 
$\tau > 0$ (see Remark~\ref{rem:nothing;conclc:stab:E2}).
It is, however, possible to prove stability in some concrete situations
of biological significance (see Remark~\ref{rem:eclipse}).

\begin{theorem}
\label{thm:stab:E2}
Assume that $\lambda c - \beta h > 0$. 
If $\beta (\lambda c - \beta h) - acd > 0$,  
then the infected equilibrium $E_2$ 
is locally asymptotically stable for $\tau = 0$. 
\end{theorem}

\begin{proof}
Let $\beta (\lambda c - \beta h) - acd > 0$. 
The characteristic equation \eqref{eq:charact:geral} 
at $E_2$ \eqref{eq:E2} is given by
\begin{equation}
\label{eq:charac:E2}
\lambda_1^3 + A \lambda_1^2 + B \lambda_1 + C 
+ \left(-{\frac { \lambda_1 \,\beta \left( \lambda_1(c\lambda - \beta\,h) 
+d(c \lambda- \beta\,h) \right)}{cd}} \right) \exp(-\lambda_1 \tau) = 0 \, ,
\end{equation}	
where
$$
A = {\frac {{\lambda}^{2}{c}^{2}\beta+{d}^{2}\lambda\,{c}^{2}-2\,\lambda\,
c{\beta}^{2}h+{\beta}^{3}{h}^{2}}{ \left( \lambda\,c-\beta\,h \right) cd}}, 
\quad B = {\frac {c\lambda\,d\beta+c\lambda\,\beta\,h-chad-{\beta}^{2}{h}^{2}}{cd}} 
\text{ and } 
C= -{\frac { \left( -\beta\,\lambda\,c+{\beta}^{2}h+acd \right) h}{c}}.
$$ 
If $\tau = 0$, then the characteristic equation \eqref{eq:charac:E2} is
$\lambda_1^3 + D \lambda_1^2 + E \lambda_1 + F = 0$
with $D = {\frac {d\lambda\,c}{\lambda\,c-\beta\,h}}  > 0$, 
$E = {\frac { \left( \beta\,\lambda\,c - acd+d{\beta}^{2}
-{\beta}^{2}h \right) h}{cd}} > 0$, 
$F = {\frac { \left( \beta\,\lambda\,c-{\beta}^{2}h-acd \right) h}{c}} > 0$ 
and $D E - F = {\frac {h\beta\, \left( c\lambda\,d\beta
+ h(c\lambda\,\beta-cad-{\beta}^{2}{h}) \right) }{\left( 
\lambda\,c-\beta\,h \right) c}} > 0$, 
whenever $\beta (\lambda c - \beta h) - acd > 0$. 
\end{proof}

\begin{remark}
\label{rem:nothing;conclc:stab:E2}
Let $\tau > 0$. Suppose that \eqref{eq:charac:E2} 
has pure imaginary roots $\pm w i$. Replacing 
$\lambda_1$ in \eqref{eq:charac:E2} by $w i$, 
and separating the real and imaginary parts, we obtain that
\begin{equation*}
\begin{cases}
G {w}^{2}+H = J {w}^{2}+ K w\\[0.2cm]
L {w}^{3}+ M w = N {w}^{2}+ P w,
\end{cases}
\end{equation*}
where
\begin{equation*}
\begin{split}
G &=  -{\beta}^{3}{h}^{2}-\beta\,{c}^{2}{\lambda}^{2}
-{c}^{2}{d}^{2}\lambda+2\,{\beta}^{2}ch\lambda \, ,\\
H &= d{\lambda}^{2}{c}^{2}h\beta+{h}^{2}\beta\,{d}^{2}ac
-{d}^{2}\lambda\,{c}^{2}ha-2\,d\lambda\,c{\beta}^{2}{h}^{2}
+{h}^{3}{\beta}^{3}d \, ,\\
J &= \beta\, \left( {\beta}^{2}{h}^{2}\cos \left( w\tau \right) 
+{\lambda}^{2}{c}^{2}\cos\left( w\tau \right) 
-2\,\lambda\,c\beta\,h\cos \left(w\tau \right)\right) \, ,\\
K &= \beta\, \left( -d{\lambda}^{2}{c}^{2}\sin \left( 
w\tau \right) +2\,d\lambda\,c\beta\,h\sin \left( w\tau\right) 
-{\beta}^{2}{h}^{2}d\sin \left( w\tau \right)  \right) \, ,\\
L &= -\lambda\,{c}^{2}d+\beta\,hcd \, , \\ 
M &=  {\beta}^{3}{h}^{3}-\lambda\,{c}^{2}had-d\lambda\,c{\beta}^{2}h
+\beta\,{h}^{2}a cd-2\,\lambda\,c{\beta}^{2}{h}^{2}
+d{\lambda}^{2}{c}^{2}\beta+{\lambda}^{2}{c}^{2}h\beta \, ,\\
N &=  -\beta\, \left( \sin \left( w\tau \right) {c}^{2}{\lambda}^{2}
-2\,\sin\left( w\tau \right) \beta\,ch\lambda
+\sin \left( w\tau \right){\beta}^{2}{h}^{2} \right) \, , \\
P &= -\beta\, \left( \cos \left( w\tau\right) {\beta}^{2}d{h}^{2}
-2\,\cos \left( w\tau \right) \beta\,cdh\lambda
+\cos\left( w\tau \right) {c}^{2}d{\lambda}^{2} \right)  \, .
\end{split}
\end{equation*} 
By adding up the squares of both equations, 
and using the fundamental trigonometric formula, we obtain that
\begin{equation}
\label{eq:caract:E2:taupos}
Q {w}^{6}+ R {w}^{4}+ S {w}^{2} + T = 0,
\end{equation}
where
\begin{equation*}
\begin{split}
Q &=  {\lambda}^{2}{c}^{4}{d}^{2}+{\beta}^{2}{h}^{2}{c}^{2}{d}^{2}
-2\,\lambda\,{c}^{3}{d}^{2}\beta\,h, \\
R &= 6\,{\lambda}^{2}{c}^{3}d{\beta}^{2}{h}^{2}
+2\,{\lambda}^{2}{c}^{4}{d}^{2}ha+{d}^{4}{\lambda}^{2}{c}^{4}
+2\,{\beta}^{4}{h}^{4}cd-2\,{\lambda}^{3}{c}^{4}dch\beta
+2\,{\beta}^{2}{h}^{3}a{c}^{2}{d}^{2}-6\,{\beta}^{3}{h}^{3}
\lambda\,{c}^{2}d-4\,\lambda\,{c}^{3}{d}^{2}\beta\,{h}^{2}a,\\
S &=  -2\,{\beta}^{6}{h}^{5}d-{\beta}^{6}{d}^{2}{h}^{4}
-2\,{d}^{3}{\lambda}^{3}{c}^{4}h\beta
-2\,{h}^{3}{\beta}^{3}{d}^{3}\lambda\,{c}^{2}
-2\,{\beta}^{4}{h}^{4}{d}^{2}ac+2\,{d}^{4}{\lambda}^{2}{c}^{4}ha\\
&+4\,{d}^{3}{\lambda}^{2}{c}^{3}{\beta}^{2}{h}^{2}
+4\,{\beta}^{5}{h}^{3}{d}^{2}\lambda\,c-6\,{\beta}^{3}{h}^{4}\lambda\,{c}^{2}ad
-2\,{\lambda}^{2}{c}^{3}{h}^{2}a{d}^{2}{\beta}^{2}
-2\,\lambda\,{c}^{3}{h}^{3}{a}^{2}{d}^{2}\beta\\
&+6\,{\lambda}^{2}{c}^{3}{h}^{3}ad{\beta}^{2}
-2\,{\lambda}^{3}{c}^{4}{h}^{2}ad\beta+4\,{d}^{2}
\lambda\,{c}^{2}{\beta}^{3}{h}^{3}a+{\beta}^{6}{h}^{6}
-4\,{\beta}^{5}{h}^{5}\lambda\,c
+6\,{\beta}^{4}{h}^{4}{\lambda}^{2}{c}^{2}\\
&-4\,{\lambda}^{3}{c}^{3}{\beta}^{3}{h}^{3}
+{\lambda}^{4}{c}^{4}{h}^{2}{\beta}^{2}
-2\,{h}^{2}\beta\,{d}^{4}a{c}^{3}\lambda+6\,{\beta}^{5}{h}^{4}d\lambda\,c
+2\,{\beta}^{4}{h}^{5}acd-6\,{\beta}^{4}{h}^{3}d{\lambda}^{2}{c}^{2}\\
&+{\lambda}^{2}{c}^{4}{h}^{2}{a}^{2}{d}^{2}
-5\,{d}^{2}{\lambda}^{2}{c}^{2}{\beta}^{4}{h}^{2}
+2\,{d}^{2}{\lambda}^{3}{c}^{3}{\beta}^{3}h
+2\,d{\lambda}^{3}{c}^{3}{\beta}^{3}{h}^{2}
+{\beta}^{2}{h}^{4}{a}^{2}{c}^{2}{d}^{2},\\
T &= -4\,{d}^{2}{\lambda}^{3}{c}^{3}{\beta}^{3}{h}^{3}
-4\,{d}^{2}\lambda\,c{\beta}^{5}{h}^{5}
+{d}^{2}{\lambda}^{4}{c}^{4}{h}^{2}{\beta}^{2}
-2\,{h}^{3}\beta\,{d}^{4}{a}^{2}{c}^{3}\lambda
-6\,{h}^{4}{\beta}^{3}{d}^{3}a{c}^{2}\lambda
+{h}^{6}{\beta}^{6}{d}^{2}\\
&+{h}^{4}{\beta}^{2}{d}^{4}{a}^{2}{c}^{2}+2\,{h}^{5}{\beta}^{4}{d}^{3}ac
+6\,{h}^{3}{\beta}^{2}{d}^{3}a{c}^{3}{\lambda}^{2}-2\,{d}^{3}{
\lambda}^{3}{c}^{4}{h}^{2}a\beta+{d}^{4}{\lambda}^{2}{c}^{4}{h}^{2}{a}^{2}
+6\,{d}^{2}{\lambda}^{2}{c}^{2}{\beta}^{4}{h}^{4} \, .
\end{split}
\end{equation*}
This equation admits at least two pure imaginary roots. Indeed, 
let $\lambda= 1$, $d = \frac{1}{10}$, $\beta = \frac{1}{2}$, $a = \frac{1}{5}$, 
$p = 1$, $c =\frac{1}{10}$ and $h = \frac{1}{10}$. Then, 
$\lambda c - \beta h = \frac{1}{20} > 0$, $\beta(\lambda c - \beta h) - acd 
=  \frac{23}{1000} > 0$ and equation \eqref{eq:caract:E2:taupos} is given by
$$
{w}^{6}-{\frac {21}{50}}\,{w}^{4}+{\frac {1731}{5000}}\,{w}^{2}
+{\frac {529}{1000000}} = 0 \, .
$$
This equation admits two pure imaginary roots given by
$$
{\frac {1/20\,i\sqrt {2}\sqrt {\sqrt [3]{174036
+4\,\sqrt {5409904729}}\left(  \left( 174036
+4\,\sqrt {5409904729} \right) ^{2/3}-3832
-28\,\sqrt [3]{174036+4\,\sqrt{5409904729}}\right)}}{\sqrt[3]{174036
+4\,\sqrt {5409904729}}}}
$$
and
$$
{\frac {-1/20\,i\sqrt {2}\sqrt {\sqrt [3]{174036
+4\,\sqrt {5409904729}} \left(  \left( 174036
+4\,\sqrt {5409904729} \right)^{2/3}-3832-28\,\sqrt [3]{174036
+4\,\sqrt {5409904729}} \right) }}{\sqrt [3]{174036
+4\,\sqrt {5409904729}}}} \, .
$$
Therefore, from Rouch\'e theorem, we cannot conclude anything
about the stability of the CTL equilibrium $E_2$.
\end{remark}

\begin{remark}
\label{rem:eclipse}
According with different studies, the eclipse phase 
represented by the time delay $\tau$ can take from 7 to 21 days  
\cite{Busch,Cohen,Coombs,Kahn}. Based on this, let us assume 
$\tau=10$ days. In Section~\ref{sec:numerical} we show, 
numerically, that the infected equilibrium $E_2$ 
is locally asymptotically stable for $\tau = 10$ 
and the parameter values from Table~\ref{table:parameter} 
with $\beta = 0.5$. This is easy to show analytically:
the characteristic equation \eqref{eq:charac:E2} 
in this case is given by $q(\lambda_1) = 0$ with
$$
q(\lambda_1) = -1/4\,{{\rm e}^{-\lambda_1\,\tau}}\lambda_1
-\frac{5}{2}\,{{\rm e}^{-\lambda_1\,\tau}}{\lambda_1}^{2}
+{\frac{23}{1000}}+{\frac {73\,\lambda_1}{100}}
+{\frac {27\,{\lambda_1}^{2}}{10}}+{\lambda_1}^{3},
$$ 
$q(0) = 23/1000$, and the derivative 
is always positive for $\lambda_1 \geq 0$. 
Therefore, $q(\lambda_1)$ does not have nonnegative real roots.
Analogously, we can show that $E_2$ is locally asymptotically 
stable for other positive values of the time delay $\tau$. 
This will be considered in Section~\ref{sec:numerical}. 
\end{remark}

% ---------------------

\section{Optimal control problem with state and control delays}
\label{sec:OC}

It is interesting to introduce drug therapy into the model by assuming that treatment
reduces the rate of viral replication, expressed by $(1-u)\beta xy$,
where $0 \leq  u \leq 1$. Our aim is to find a treatment strategy $u(t)$ 
that maximizes the number of $CD4^{+}$ T cells $x$
as well as the number of CTL immune response cells $z$, keeping the cost, 
measured in terms of chemotherapy strength and a combination 
of duration and intensity, as low as possible. 
Due to the importance of the pharmacological delay 
in the HIV treatment, we consider a discrete time delay in the control 
variable $u(t)$, denoted by $\xi$, which represents the delay that occurs
between the administration of drug and its appearance within cells, 
due to the time required for drug absorption, distribution, 
and penetration into the target cells \cite{PerelsonEtAllHIVdynamicsScience1996}.
Precisely, we propose the following control system with discrete time delays 
in state and control variables:
\begin{equation}
\label{eq:model:delay:control}
\begin{cases}
\dot{x}(t)={\lambda}-dx(t)-(1-u(t-\xi)){\beta}x(t)y(t),\\
\dot{y}(t)=(1-u(t-\xi)){\beta}x(t-\tau)y(t-\tau)-ay(t)-py(t)z(t),\\
\dot{z}(t)=cx(t)y(t)z(t)-hz(t).
\end{cases}
\end{equation}
The initial conditions for the state variable $z$ and, due to the delays, 
initial functions for the state variables $x$ and $y$ and control $u$, 
are given by
\begin{equation}
\label{eq:initcond:delays}
\begin{split}
z(0) &= z_0 \geq 0,\\ 
x(t) &\equiv x_0 \geq 0 \text{ for } -\tau \leq t \leq 0,\\  
y(t) &\equiv y_0 \geq 0 \text{ for } -\tau \leq t \leq 0, \\
u(t) &\equiv u_0, \text{ where } u_0 \in [0, 1],  \text{ for } -\xi \leq t < 0.
\end{split}
\end{equation}
The set of admissible control functions is given by
\begin{equation*}
\Theta = \biggl\{ u(\cdot) \in L^1\left([0, t_f], \R \right) 
\, | \,  0 \leq u(t) \leq 1 \,   
\, \forall \, t \in [0, t_f] \, \biggr\} 
\end{equation*}
and the objective functional is
\begin{equation}
\label{costfunction}
J(u(\cdot)) = \int_0^{t_f} 
\left[ x(t) + z(t) - u(t) \right] dt, 
\end{equation}
which measures the concentration of $CD4^{+}$ T and CTL cells
and the cost measured in terms of chemotherapy strength 
and a combination of duration and intensity. The optimal control
problem consists in determining a control function 
$u(\cdot) \in \Theta$ that maximizes 
the cost functional \eqref{costfunction} subject to the control system 
\eqref{eq:model:delay:control} and initial conditions \eqref{eq:initcond:delays}. 
In \cite{CulshawRuanSpiteri2004}, the authors consider a different $L^2$ 
cost functional for a non-delayed control system.  
We claim that our delayed control system \eqref{eq:model:delay:control} describes 
better the reality. Moreover, as we shall see, the extremals obtained with 
our $L^1$ cost functional \eqref{costfunction} are easier to implement 
from a medical point of view. 

We apply the optimality conditions given by the Pontryagin Maximum Principle 
for multiple delayed optimal control problems of G\"ollmann and Maurer
\cite[Theorem~3.1]{Goellmann-Maurer-14}. For that,
we introduce the delayed state variables 
$\zeta(t) =x(t - \tau)$, $\eta(t) = y(t - \tau)$ 
and the control variable $v(t) = u(t - \xi)$.
Using the adjoint variable $\psi = \left( \lambda_x, 
\lambda_y, \lambda_z \right) \in \R^3$, 
the Hamiltonian for the cost functional 
\eqref{costfunction} and the control system 
\eqref{eq:model:delay:control} is given by
\begin{equation*}
H(x, \zeta, y, \eta, z, \psi, u, v) 
= x + z - u + \lambda_x \left( \lambda-d x -(1-v)\beta x y \right)
+ \lambda_y \left((1-v)\beta \zeta \eta - a y - p y z \right) 
+ \lambda_z \left( c xyz - h z \right).  
\end{equation*}
The adjoint equations are given by 
\begin{equation*}
\begin{cases}
\dot{\lambda}_x(t) = - H_x[t] 
- \chi_{[0, t_f-\tau]} H_\zeta[t + \tau],\\
\dot{\lambda}_y(t) = - H_y[t]  
- \chi_{[0, t_f-\tau]} H_\eta[t + \tau],\\
\dot{\lambda}_z(t) = - H_z[t],
\end{cases}
\end{equation*}
where the subscripts denote partial derivatives and 
$\chi_{[0, t_f-\tau]}$ is the characteristic function 
on the interval $[0, t_f - \tau]$ (see \cite{Goellmann-Maurer-14}). 	
Since the terminal state is free, i.e.,
$(x(t_f), y(t_f), z(t_f)) \in \R^3$,  
the transversality conditions
$$ 
\lambda_x(t_f)=\lambda_y(t_f) = \lambda_z(t_f) = 0
$$	
hold. To characterize the optimal control $u$, we introduce 
the following \emph{switching function}:
\begin{equation*}
\begin{split}
\phi(t) &= H_u[t] + \chi_{[0, t_f - \xi]} H_v[t + \xi]\\
&= \begin{cases}
-1 + \lambda_x(t +\xi) \beta x(t + \xi) y(t + \xi) 
- \lambda_y \beta \zeta(t + \xi) \eta(t + \xi)  
\quad \text{for} \quad 0 \leq t \leq t_f - \xi,\\
-1 \quad \text{for} \quad t_f - \xi \leq t \leq t_f.
\end{cases}
\end{split}	
\end{equation*}	
The maximality condition of the Pontryagin Maximum Principle 
\cite[Theorem~3.1]{Goellmann-Maurer-14} gives the control law
\begin{equation}
\label{control-law}
u(t) = \left\{
\begin{array}{rcl}
1 &&\mbox{if} \quad \phi(t) > 0,  \\[1mm]
0 && \mbox{if} \quad \phi(t) < 0,  \\[1mm]
{\rm singular} &&  \mbox{if} \quad \phi_k(t) = 0
\;\; \mbox{on} \; I_s \subset [0,t_f].
\end{array}
\right.
\end{equation}	
	
% ---------------------------------------------

\section{Numerical simulations}
\label{sec:numerical}	

We begin by showing numerically, in Section~\ref{sec:stab:EP},
the local stability of the equilibrium points
$E_0$ and $E_2$ that was proved in Section~\ref{sec:3}.
Then, in Section~\ref{subsec:OC}, we apply the necessary conditions
of optimal control of Section~\ref{sec:OC} to a situation 
of fast convergence to the CTL equilibrium $E_2$.

% -------------------------------	

\subsection{Stability of the equilibrium points}	
\label{sec:stab:EP}

Following \cite{CulshawRuanSpiteri2004,WodarzNowak}, 
we consider the parameter values of Table~\ref{table:parameter}. 
% ---------------------------
\begin{table}[htbp]
\begin{center}
\begin{tabular}{c|l|r} \hline
Parameter & Description &  Value \\[0.1cm] \hline
$\lambda$ & source rate of $CD4+ T$ cells  & $1$ $cells/day$ \\
$d$& decay rate of $CD4+ T$ cells   & $0.1 $ $cells/day$\\
$\beta$& rate $CD4+ T$ cells become infected & $[0.00025, 0.5]$  $cells/day$\\
$a$& death rate infected, not by CTL killing   & $0.2$  $cells/day$\\
$p$& rate at which infected cells are killed by CTLs   & $1/day$  \\
$c$& immune response activation rate & $0.1/day$\\
$h$& death rate of CTLs  & $0.1/day$\\ \hline			
\end{tabular}
\caption{Parameter values.}
\label{table:parameter}
\end{center}
\end{table}
% ---------------------------

Two different initial conditions for the state variable $z$ and, 
due to the delays, initial functions for the state variables 
$x$ and $y$, are considered: 
\begin{equation}
\label{eq:initcond:delays:numer}
\begin{split}
x(t) &\equiv 45, \quad y(t) \equiv 3, \quad -\tau \leq t \leq 0,\\ 
z(0) &= 20,
\end{split}
\end{equation}
and 
\begin{equation}
\label{eq:initcond:delays:numer:Culschaw}
\begin{split}
x(t) &\equiv 5, \quad y(t) \equiv 1, \quad -\tau \leq t \leq 0,\\ 
z(0) &= 2.
\end{split}
\end{equation}
Consider the parameter values of Table~\ref{table:parameter} and $\beta = 0.00025$. 
For these parameter values, we have $\beta \lambda - d a = -0.0208 < 0$. 
Let the time delay $\tau$ be equal to 10 days: $\tau = 10$.  In Figure~\ref{fig:stab:E0} 
we observe the convergence of the variables $x$, $y$, $z$ to the steady state 
$E_0 = \left(\frac{\lambda}{d}, 0, 0 \right) = \left(10, 0, 0 \right)$. 
% -----------------------------------------------------------------
\begin{figure}[htb]
\centering
\subfloat[\footnotesize{$x(t)$, $t \in [0, 500]$}]{\label{Stab:E0:x}
\includegraphics[width=0.45\textwidth]{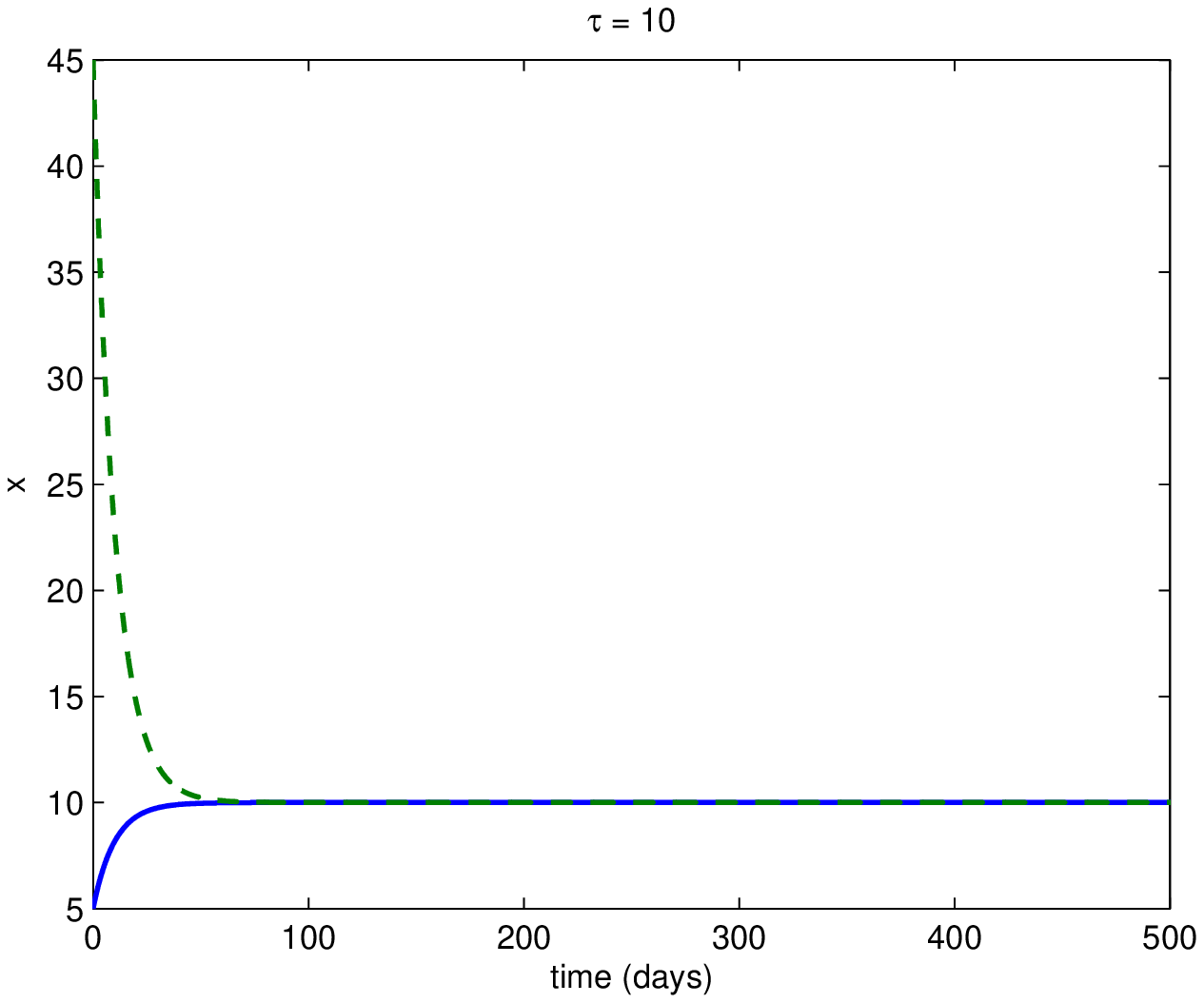}}
\subfloat[\footnotesize{$(y(t), z(t)$, $t\in [0, 500]$}]{\label{Stab:E0:yz}
\includegraphics[width=0.45\textwidth]{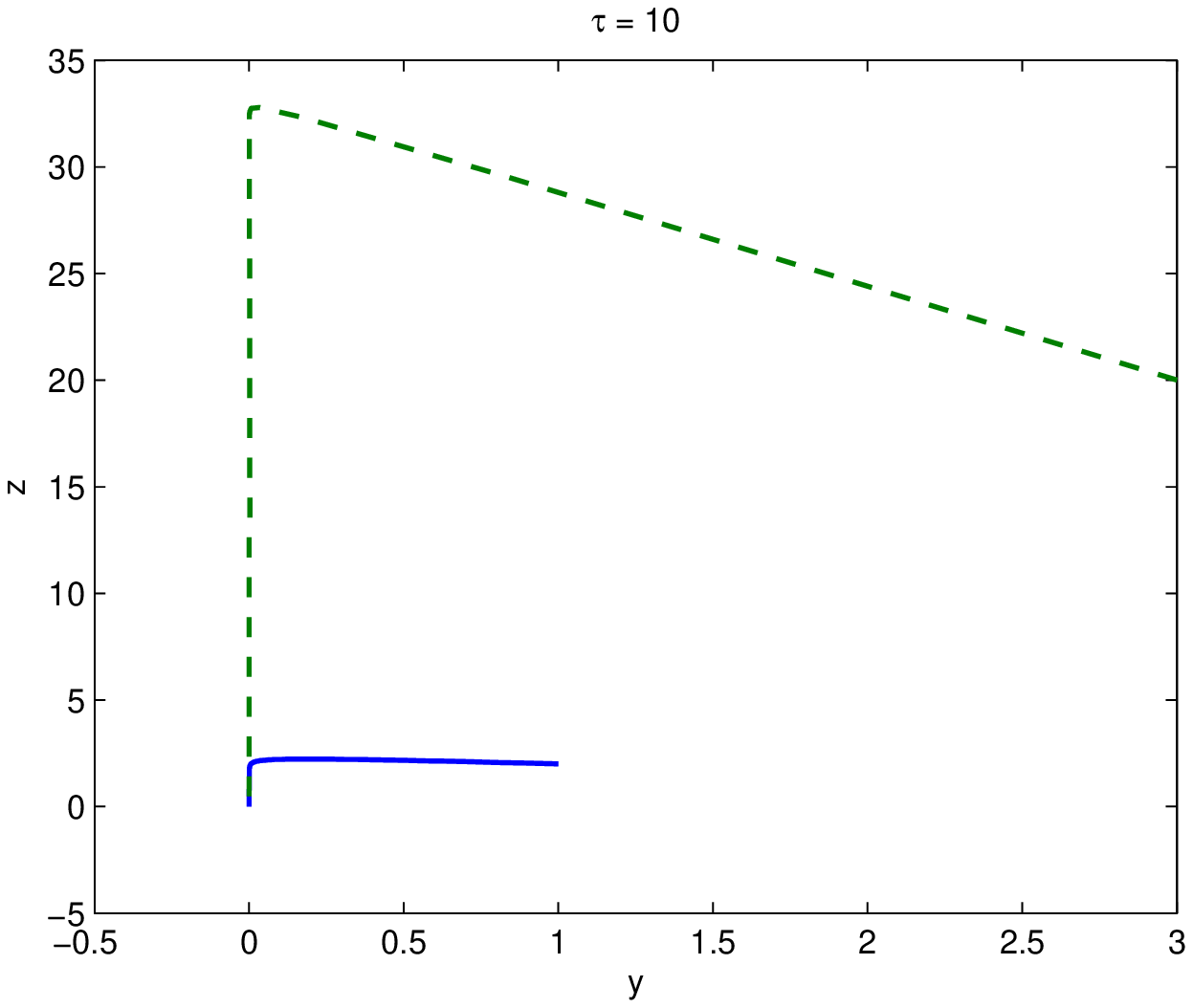}}
\caption{Infection-free equilibrium $E_0$ for parameter values given 
in Table~\ref{table:parameter}, $\beta = 0.00025$ and time delay $\tau=10$ days. 
The dashed line corresponds to the initial conditions \eqref{eq:initcond:delays:numer} 
and the continuous line corresponds to the initial conditions 
\eqref{eq:initcond:delays:numer:Culschaw}.}
\label{fig:stab:E0}
\end{figure}
% -----------------------------------------------------------------

The initial conditions \eqref{eq:initcond:delays:numer:Culschaw} are closer 
to the infected equilibrium point $E_2$ for the parameter values 
of Table~\ref{table:parameter} and $\beta = 0.5$. For these parameter values, 
one has $\beta \lambda - d a = 0.48 > 0$, $\lambda c - \beta h = 0.05 > 0$ 
and $\beta (\lambda c - \beta h) - acd = 0.023 > 0$. In Figure~\ref{fig:stab:E2:2CI}, 
we observe the convergence of the variables $x$, $y$, $z$ to the steady state 
$E_2=\Biggl(  \frac{\lambda c - \beta h}{c d}, \frac{d h}{ \lambda c - \beta h}, 
\frac{\beta (\lambda c- \beta h)}{c d p} - \frac{ a}{p} \Biggr) 
= \left( 5, 0.2, 2.3  \right)$ by considering the initial conditions 
\eqref{eq:initcond:delays:numer} and \eqref{eq:initcond:delays:numer:Culschaw} 
and $\tau = 10$. 
% -----------------------------------------------------------------
\begin{figure}[htb]
\centering
\subfloat[\footnotesize{$x(t)$, $t \in [0, 500]$}]{\label{Stab:E2:x}
\includegraphics[width=0.33\textwidth]{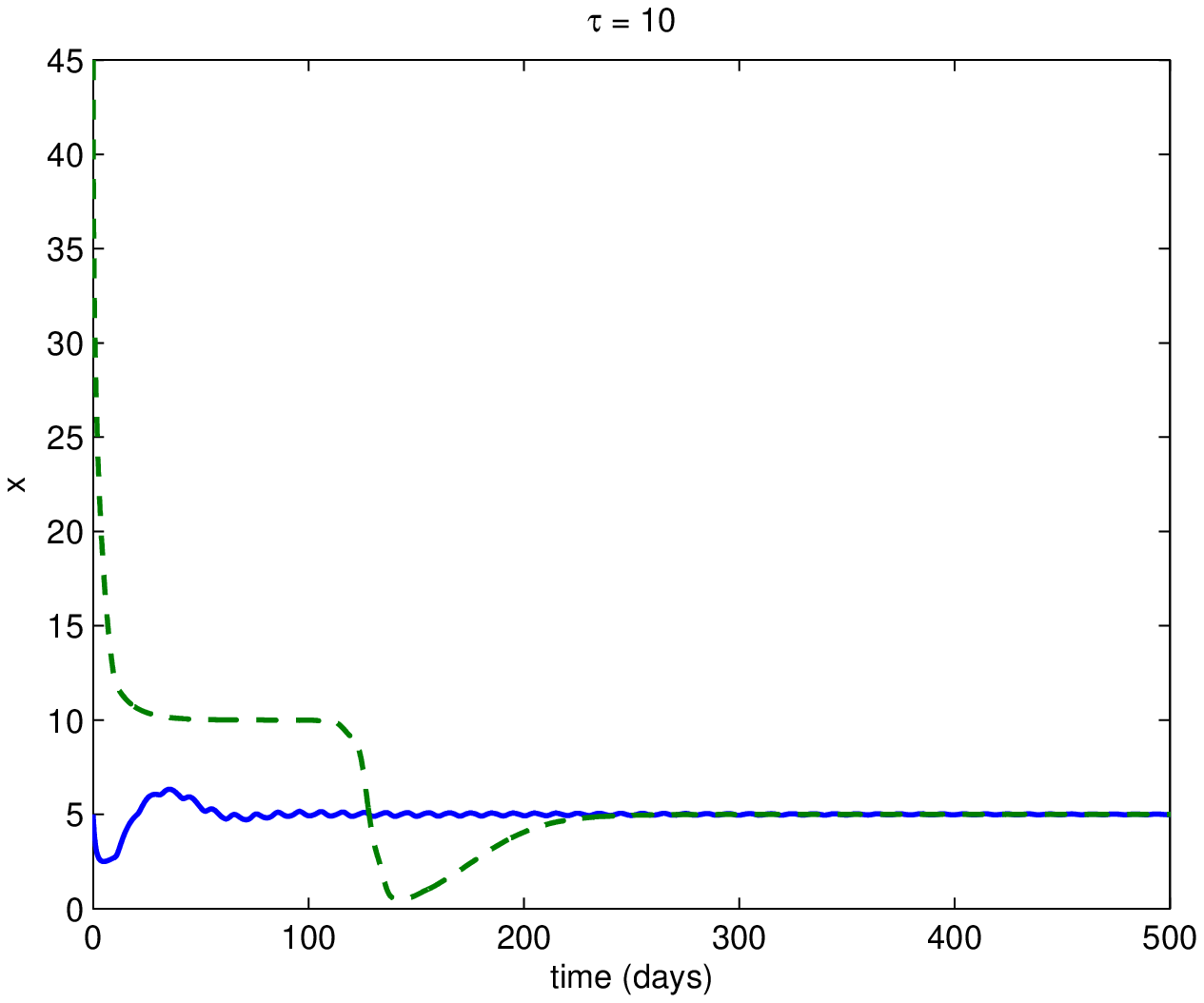}}
\subfloat[\footnotesize{$y(t)$, $t \in [0, 500]$}]{\label{Stab:E2:y}
\includegraphics[width=0.33\textwidth]{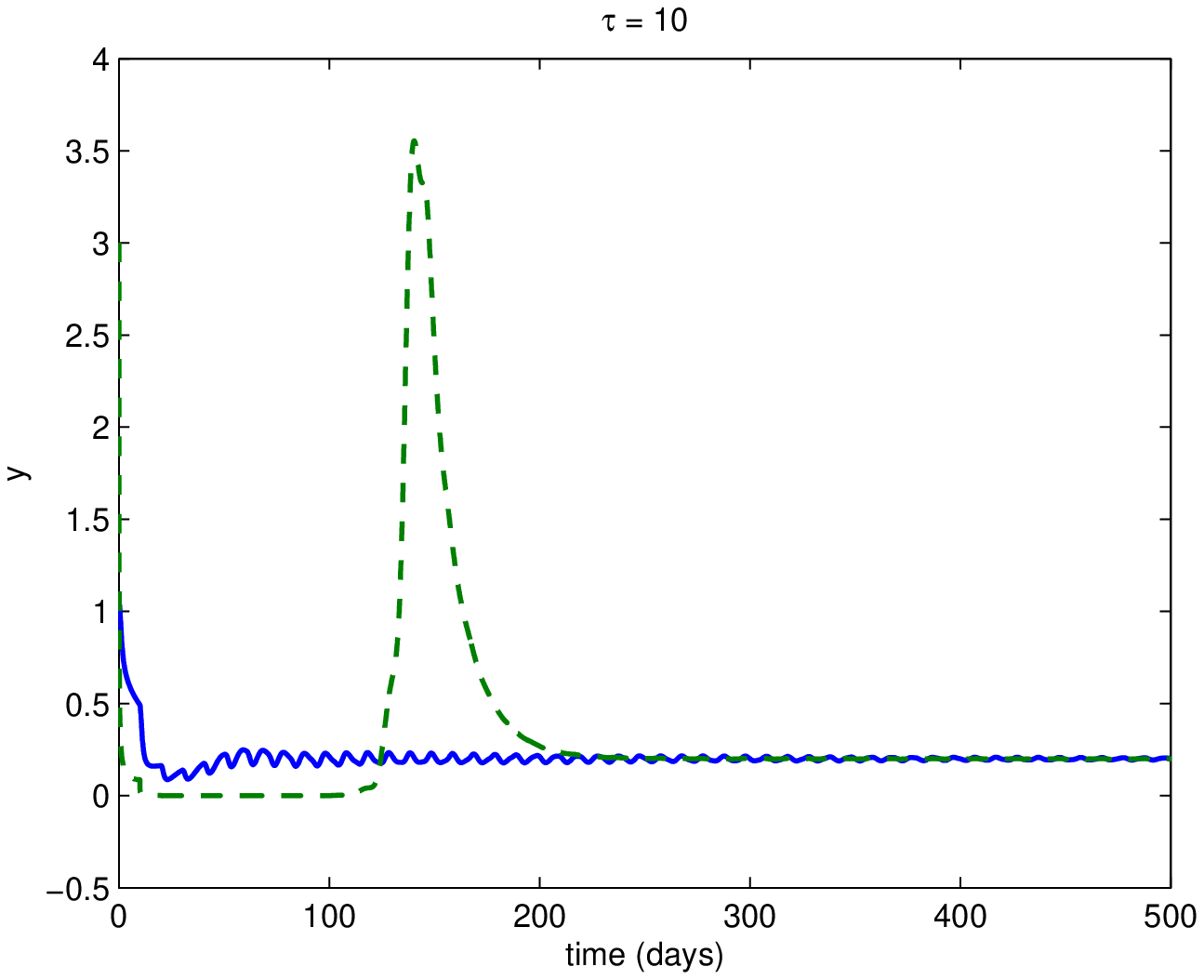}}
\subfloat[\footnotesize{$z(t)$, $t \in [0, 500]$}]{\label{Stab:E2:z}
\includegraphics[width=0.33\textwidth]{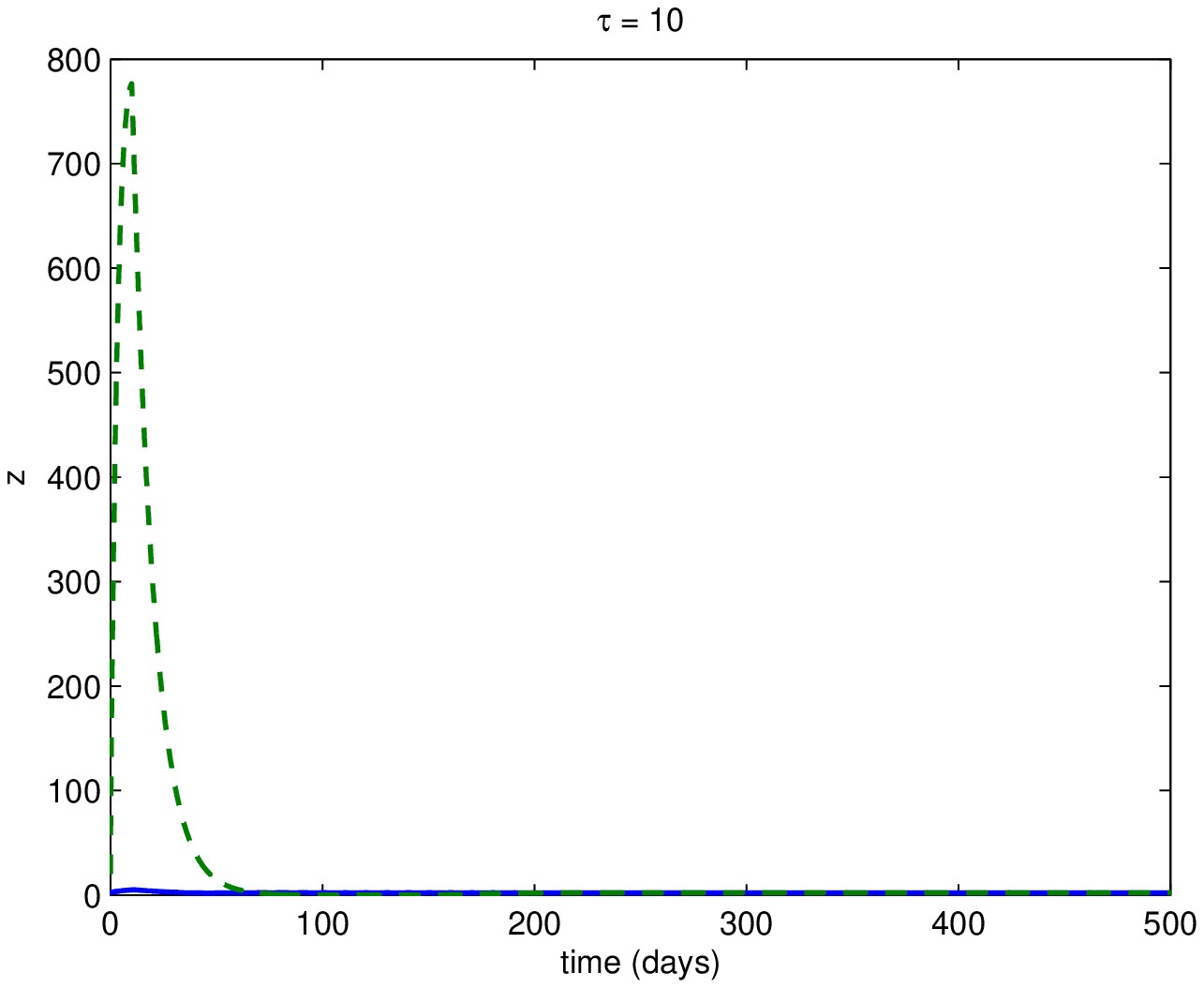}}
\caption{Endemic CTL equilibrium $E_2$ for the parameter values given 
in Table~\ref{table:parameter}, $\beta = 0.5$ and time delay $\tau=10$. 
The dashed line corresponds to the initial conditions 
\eqref{eq:initcond:delays:numer} and the continuous line 
corresponds to the initial conditions 
\eqref{eq:initcond:delays:numer:Culschaw}.}
\label{fig:stab:E2:2CI}
\end{figure}
% -------------------------------	

There are situations where the convergence to stability is much slower.
This is illustrated in Figure~\ref{fig:fastconv:tau1:tau10}: 
slower convergence for $\tau = 10$ versus faster convergence
for $\tau = 1$.
% -----------------------------------------------------------------
\begin{figure}[htb]
\centering
\subfloat[\footnotesize{$x(t)$, $t \in [0, 250]$}]{\label{Stab:E2:x:tau1:10}
\includegraphics[width=0.33\textwidth]{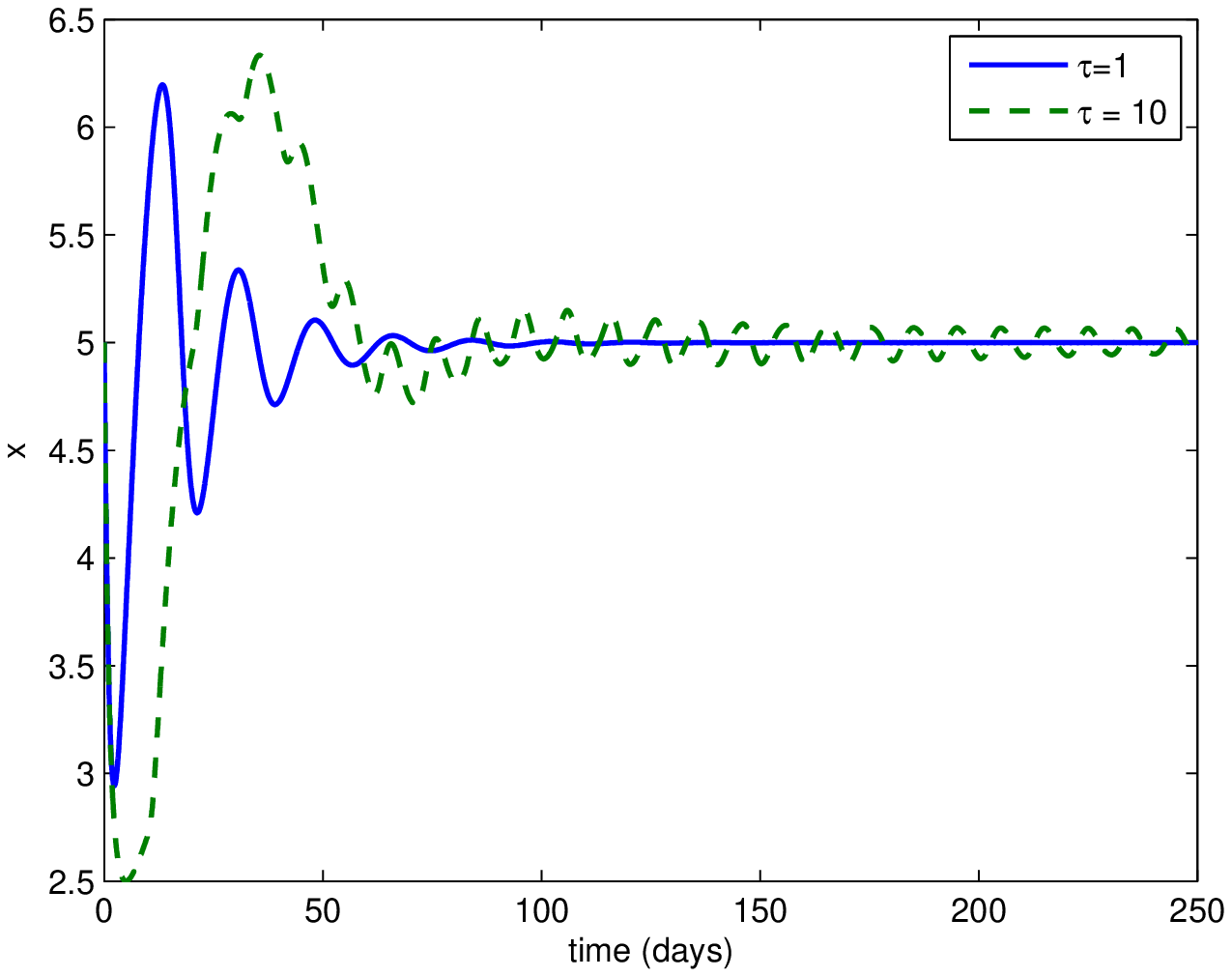}}
\subfloat[\footnotesize{$y(t)$, $t \in [0, 250]$}]{\label{Stab:E2:y:tau1:10}
\includegraphics[width=0.33\textwidth]{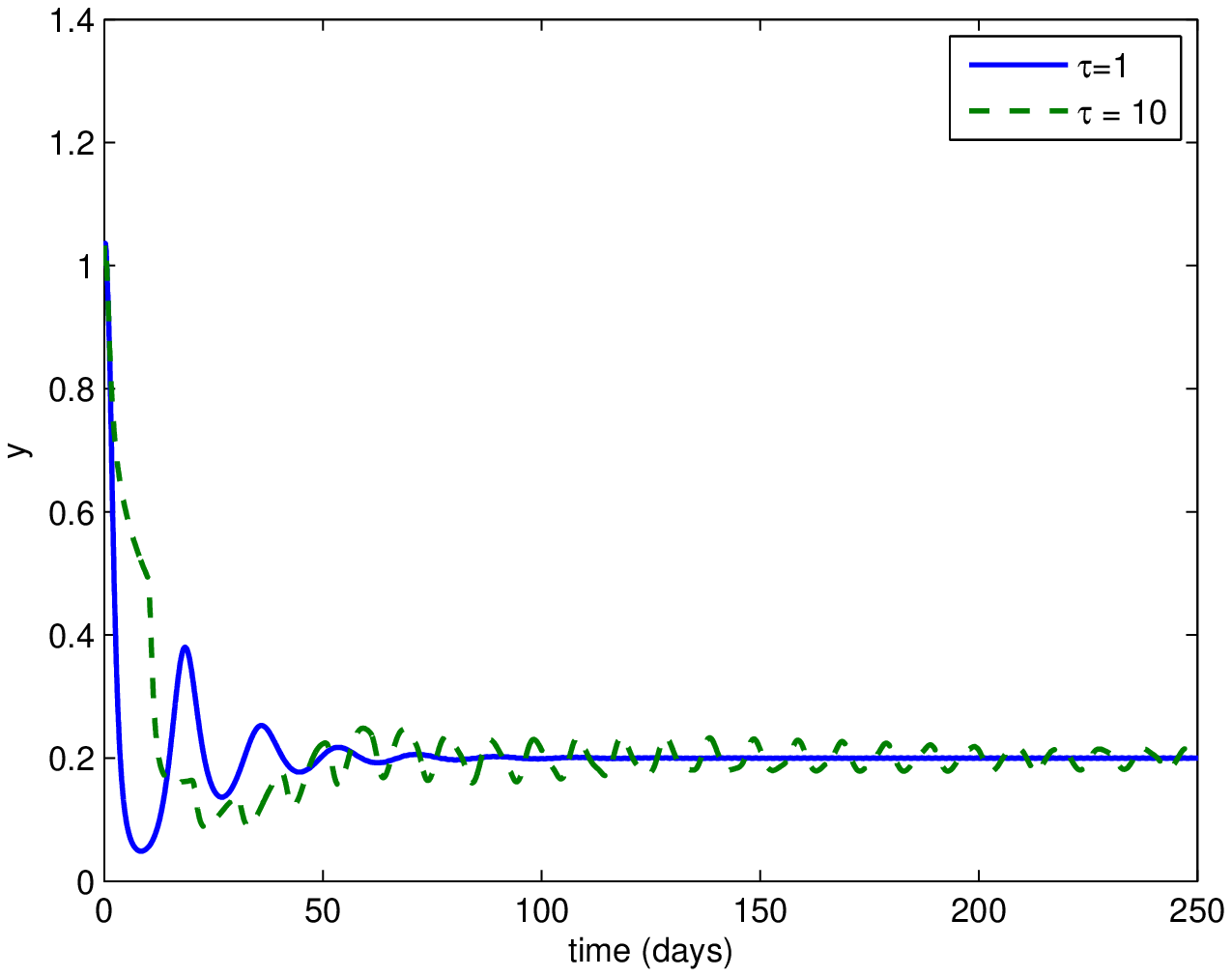}}
\subfloat[\footnotesize{$z(t)$, $t \in [0, 250]$}]{\label{Stab:E2:z:tau1:10}
\includegraphics[width=0.33\textwidth]{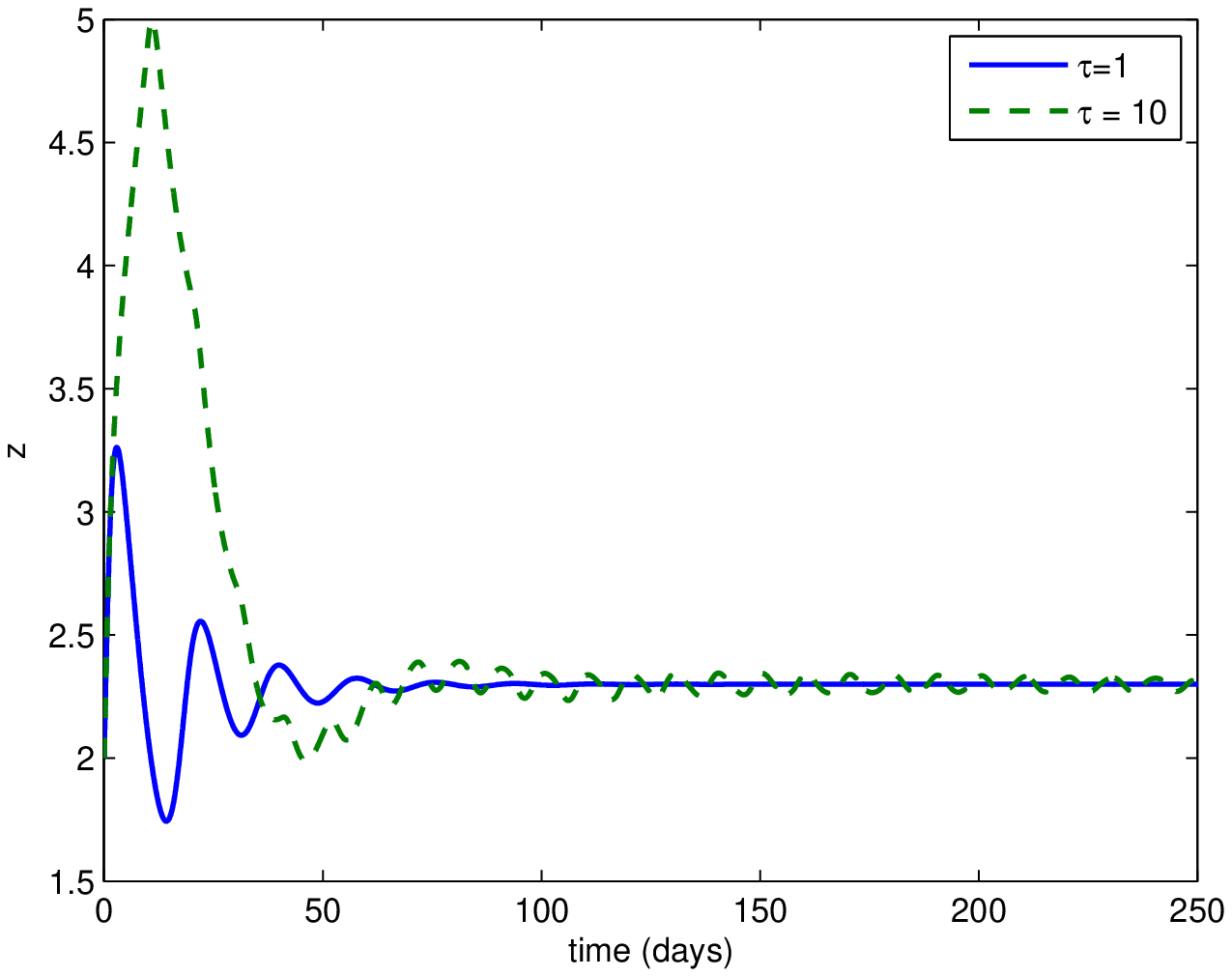}}
\caption{Endemic CTL equilibrium $E_2$ for parameter values given 
in Table~\ref{table:parameter}, $\beta = 0.5$ and initial conditions 
\eqref{eq:initcond:delays:numer:Culschaw}. The dashed line corresponds 
to $\tau = 10$ and the continuous line corresponds $\tau = 1$.}
\label{fig:fastconv:tau1:tau10}
\end{figure}

% ----------------------------------------------------------------

\subsection{Optimal control}	
\label{subsec:OC}

Let us consider the initial conditions 
\eqref{eq:initcond:delays:numer:Culschaw} 
and the initial function for the control given by
\begin{equation*}
u(0) \equiv 0, \quad -\xi \leq t < 0.
\end{equation*}
The extremal for the non-delayed (i.e., $\tau = \xi = 0$) 
optimal control problem with a $L^2$ functional was investigated 
in \cite{CulshawRuanSpiteri2004}. Figure~\ref{fig:semdelay} shows 
that the extremal control for the $L^1$ functional \eqref{costfunction}
is completely different from the $L^2$ case studied in \cite{CulshawRuanSpiteri2004}: 
with $\tau = \xi = 0$ and the $L^1$ functional \eqref{costfunction}, 
the extremal control is bang-bang with several switchings 
while the $L^2$ control extremal \cite{CulshawRuanSpiteri2004}
is singular after an initial short period of time. 
We conclude that the $L^1$ functional \eqref{costfunction} is more suitable,
from a medical point of view, because a bang-bang control 
is much easier to implement than the singular control of \cite{CulshawRuanSpiteri2004}.
% -----------------------------------------------------------------
\begin{figure}[htb]
\centering
\subfloat[\footnotesize{$x(t)$, $y(t)$ and $z(t)$}]{\label{states:semdelay}
\includegraphics[width=0.45\textwidth]{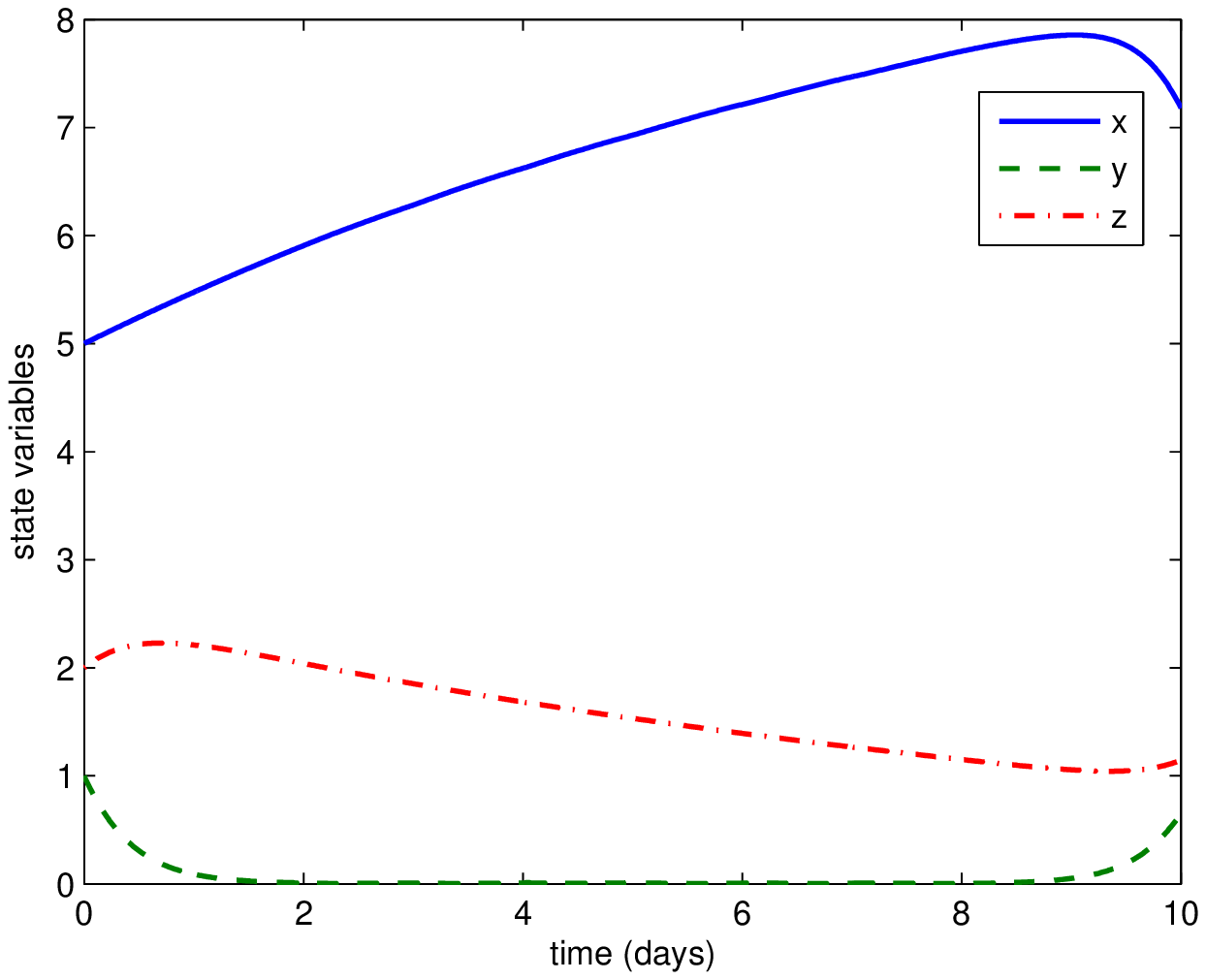}}
\subfloat[\footnotesize{$u(t)$ \eqref{control-law}}]{\label{control:semdelay}
\includegraphics[width=0.45\textwidth]{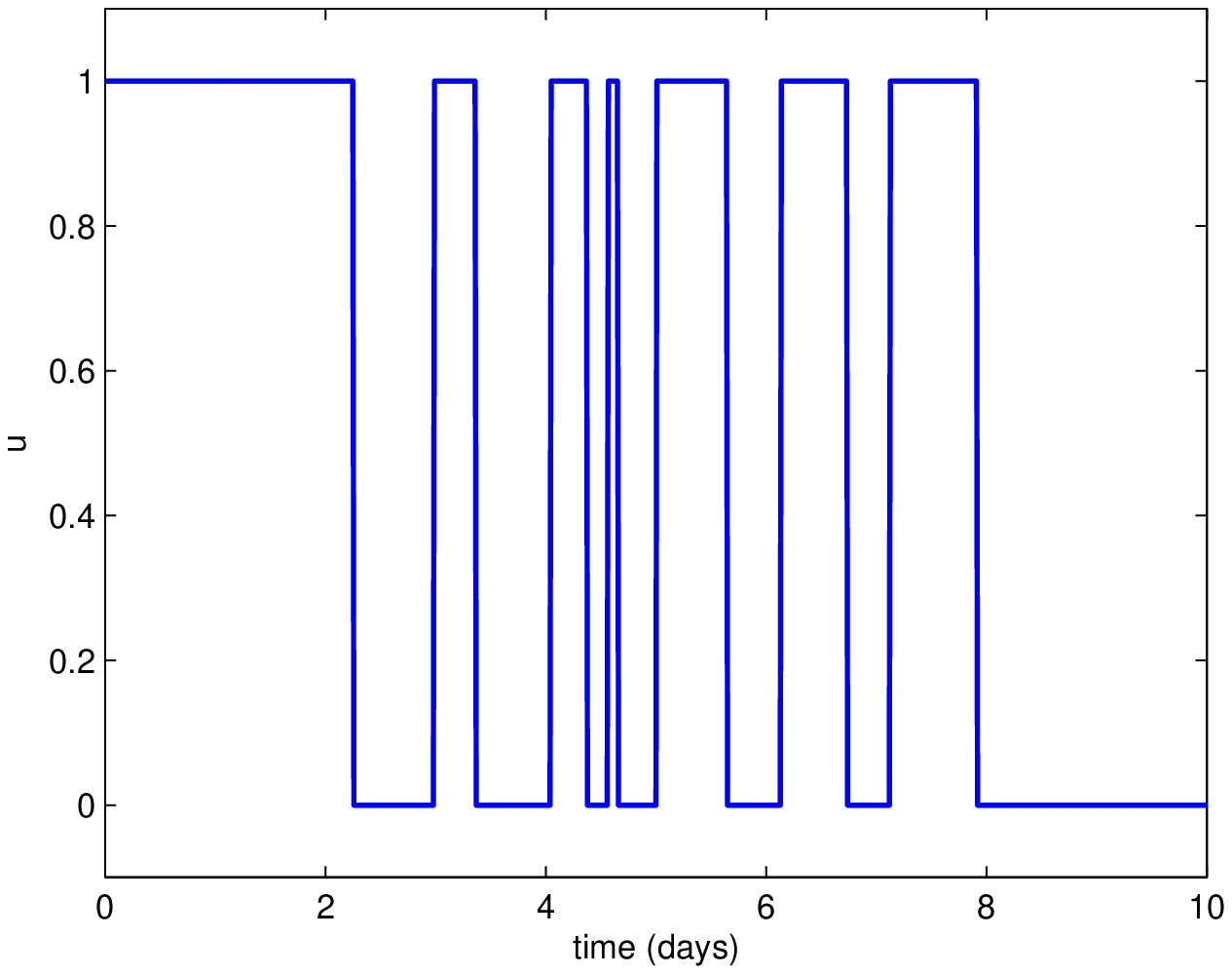}}
\caption{Extremal solutions of the optimal control problem 
for $t \in [0, 10]$ and $\tau = \xi = 0$.}
\label{fig:semdelay}
\end{figure}
% -------------------------------	
Moreover, we note that in \cite{CulshawRuanSpiteri2004} the extremal 
trajectory $z(t)$ is zero at the final time and the objective 
is to maximize $x(t)$ and $z(t)$. Our extremal 
$z(t)$ is always positive and is therefore better
than the one of \cite{CulshawRuanSpiteri2004}: 
compare our Figure~\ref{fig:semdelay} with
Figure~4.1 on page~557 of \cite{CulshawRuanSpiteri2004}.

Let us now consider an optimal control problem with both
incubation and pharmacological time delays. For illustrative
purposes, let $\tau = 0.5$ and $\xi = 0.1$. 
We see from Figure~\ref{fig:comdelay}
that the extremal state variables are similar
to the ones without delay shown in Figure~\ref{fig:semdelay}.
% -----------------------------------------------------------------
\begin{figure}[htb]
\centering
\subfloat[\footnotesize{$x(t)$, $y(t)$ and $z(t)$}]{\label{states:comdelay}
\includegraphics[width=0.45\textwidth]{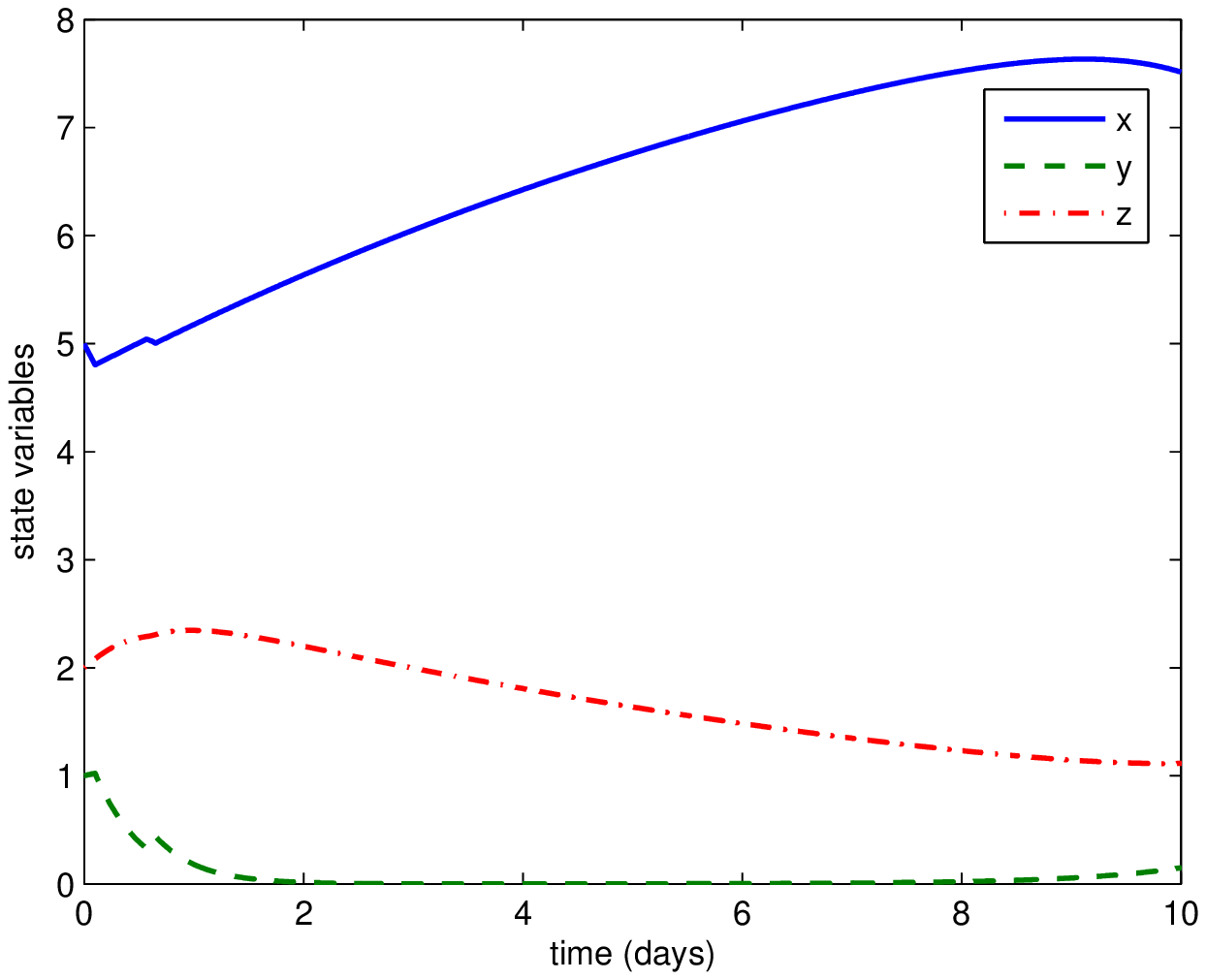}}
\subfloat[\footnotesize{$u(t)$} \eqref{control-law}]{\label{control:delay}
\includegraphics[width=0.45\textwidth]{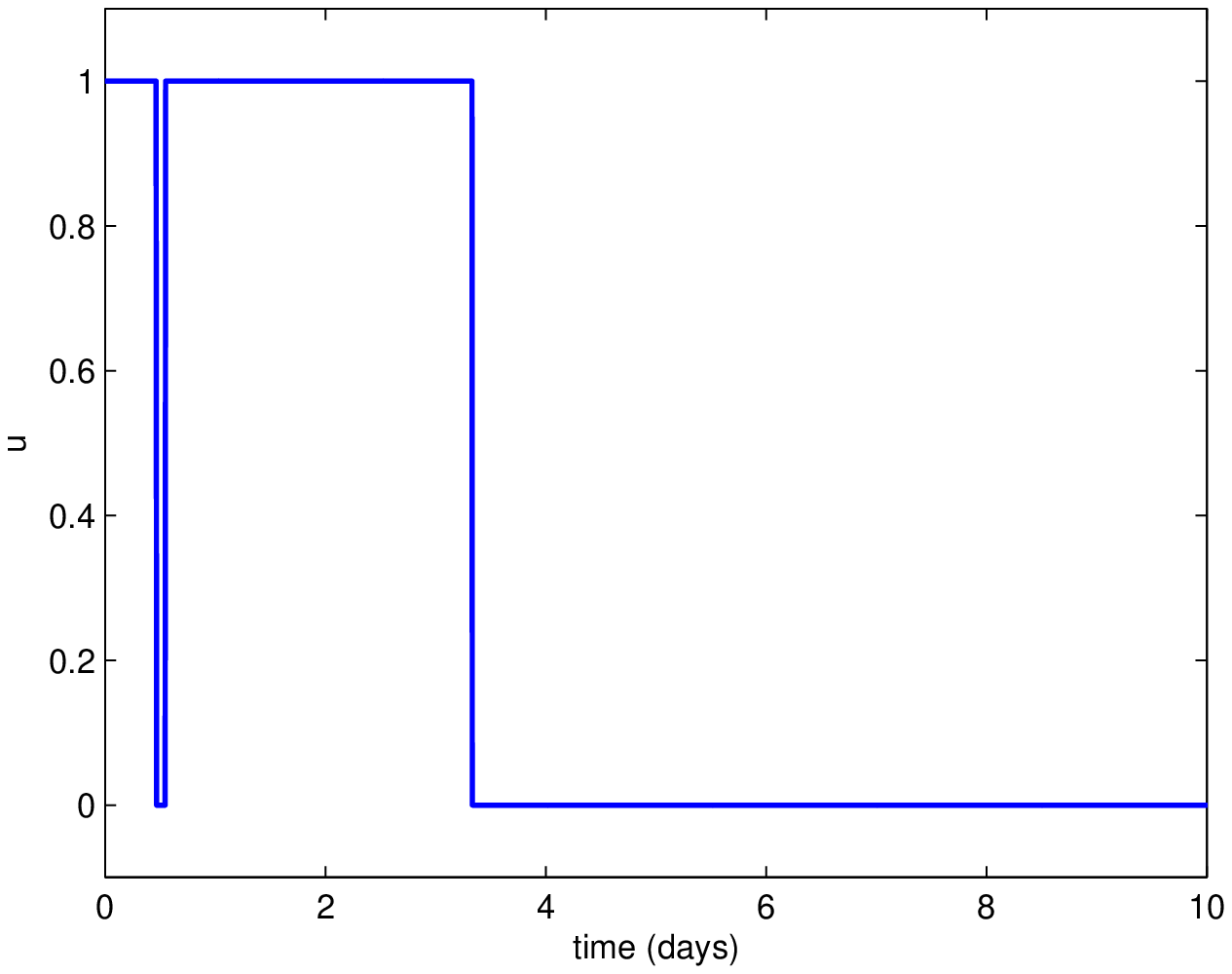}}
\caption{Extremal solutions of the optimal control problem for 
$t \in [0, 10]$, $\tau =0.5$ and $\xi = 0.1$.}
\label{fig:comdelay}
\end{figure}
% -----------------------------------------------------------------
Importantly, the number of switchings of the extremal control 
has decreased, which makes it even more simpler to implement in practice.

% -------------------------------------------

\section{Conclusion}
\label{sec:conc:fw}

We have proposed a new model for the optimal control of HIV
at cell level, which considers not only an intracellular delay 
(delay $\tau$ in the state variables) but also a pharmacological 
delay (delay $\xi$ in the control function). Local stability of the 
equilibria was investigated and the extremal control derived from 
application of the Pontryagin necessary optimality condition 
of G\"ollmann and Maurer \cite[Theorem~3.1]{Goellmann-Maurer-14}.

The extremal control for our optimal control problem, 
with the same values for the parameters as those of 
\cite{CulshawRuanSpiteri2004,WodarzNowak}, is bang-bang, 
that is, it attains alternately the boundary values 0 and 1. 
This type of control is easier to implement, from a medical point of view, 
and leads to better results than the ones previously obtained in 
\cite{CulshawRuanSpiteri2004} for a non-delayed problem with a $L^2$ functional. 

We offer to the community three open questions: (i) how to prove stability 
of the CTL equilibrium \eqref{eq:E2} for an arbitrary 
$\tau > 0$ (see Remark~\ref{rem:nothing;conclc:stab:E2});
(ii) how to prove sufficient conditions of optimality
for our problem with delays in both state and control variables; 
(iii) how to solve our optimal control problem numerically
when one increases $t_f$, $\tau$ and $\xi$. 

% -------------------------------------------

\section*{Acknowledgements}

This research was partially supported by the
Portuguese Foundation for Science and Technology (FCT)
within projects UID/MAT/04106/2013 (CIDMA) 
and PTDC/EEI-AUT/2933/2014 (TOCCATTA),
co-funded by FEDER funds through COMPETE2020 -- 
Programa Operacional Competitividade 
e Internacionaliza\c{c}\~ao (POCI) 
and by national funds (FCT).
Rocha is also supported by the FCT 
Ph.D. fellowship SFRH/BD/107889/2015; 
Silva by the FCT post-doc fellowship 
SFRH/BPD/72061/2010.

% -------------------------------------------

% -------------------------------------------

\end{document}